\documentclass[12pt]{article}
\usepackage{amssymb}
\usepackage{amsmath}
\usepackage{amsthm}
\usepackage{amsfonts}

\newcommand{\eps}{\epsilon}

\newcommand{\diff}{{\rm d}}

\newcommand{\mmp}{\mathbb{P}}

\newcommand{\me}{\mathbb{E}}
\newcommand{\mr}{\mathbb{R}}
\newcommand{\mn}{\mathbb{N}}

\newcommand{\lin}{\underset{n\to\infty}{\lim}}

\newtheorem{thm}{Theorem}[section]
\newtheorem{proposition}[thm]{Proposition}
\newtheorem{lemma}[thm]{Lemma}

\newtheorem{assertion}[thm]{Proposition}

\theoremstyle{definition}

\theoremstyle{remark}

\begin{document}
\title{Lambda-coalescents with  dust component}
\date{\today}

\author{
Alexander Gnedin\thanks{
Utrecht University,
e-mail: A.V.Gnedin@uu.nl},
\quad
Alexander Iksanov\thanks{
National T. Shevchenko University of Kiev,
e-mail: iksan@unicyb.kiev.ua}
\quad and\quad
Alexander
Marynych\thanks{
National T. Shevchenko
University of Kiev,
e-mail:
marynych@unicyb.kiev.ua} }

\maketitle
\begin{abstract}
\noindent We consider the lambda-coalescent processes with positive frequency
of singleton clusters.
The class in focus
covers, for instance, the beta$(a,b)$-coalescents with $a>1$.
We show that some large-sample properties of these processes can be derived by
coupling the coalescent with an
increasing
L{\'e}vy process (subordinator), and by
exploiting parallels with the theory of regenerative composition structures.
In particular, we discuss the limit distributions of the absorption time and the number of collisions.
\end{abstract}
\noindent Keywords: absorption time, coupling, lambda-coalescent,
number of collisions, regenerative composition structure,
subordinator.

\section{Introduction} \label{intro}

The lambda-coalescent  with values in partitions of $n$ integers
is a Markovian process
$\Pi_n=(\Pi_n(t))_{t\geq 0}$
which starts at $t=0$  with $n$
singletons and evolves according to the rule: for each $t\geq 0$
when the number of clusters is $m$, each $k$ tuple of them is
merging in one cluster at probability rate
\begin{equation}\label{lrates}
\lambda_{m,k}=\int_{0}^1 x^{k}(1-x)^{m-k}\nu({\rm d}x),
\qquad 2\leq k\leq m,
\end{equation}
where $\nu$ is a measure on the unit interval with
finite second moment.
The integral representation of rates \eqref{lrates} ensures that
the processes $\Pi_n$ can be defined consistently for all $n$, as
restrictions of a coalescent process $\Pi_\infty$
 which starts
with infinitely many clusters
and assumes values in the set of partitions of $\mn$, see \cite{Pit}.
The infinite coalescent
$\Pi_\infty$ may be regarded as a limiting form of $\Pi_n$ as
$n\to\infty$, and uniquely connected to a process with values in the
infinite-dimensional space of partitions of a unit mass.
The lambda-coalescents were introduced in the papers by Pitman \cite{Pit} and Sagitov \cite{Sag}, where the parameterization
by finite measure
$\Lambda ({\rm d}x)=x^2\nu({\rm d}x)$ was used.
The reader is referred to the recent lecture notes \cite{BerestyckiLec, BertoinLec}
for accessible introduction in the theory of  lambda-coalescents and a survey.

\par  After some number of collisions (merging events)
$\Pi_n$  enters the absorbing state with a sole cluster. Two basic
characteristics of the speed of the coalescence are {\it the
absorption time} $\tau_n$ and {\it the number of collisions}
$X_n$. The large-$n$ properties of $\tau_n$ and $X_n$ are strongly
determined by  the concentration of  measure $\nu$ on the unit interval near the endpoints of $[0,1]$.

We
suppose that $\nu$ has no mass at $1$, which excludes forced termination of $\Pi_\infty$  at independent exponential time.
The coalescent is said to come down from infinity if $\Pi_\infty(t)$ has finitely many clusters
for each $t>0$ almost surely; then  $\tau_n$ converge to a finite random variable $\tau_\infty$ which is the absorption time of $\Pi_\infty$.
Otherwise, $\Pi_\infty(t)$ almost surely stays with infinitely many clusters  for all $t$.
There is a delicate
criterion in terms of the rates $\lambda_{m,k}$ to distinguish between the two alternatives
\cite{Sch}.

In this paper we shall study $\tau_n$ and $X_n$ under the  assumption that $\Pi_\infty$ stays
infinite due to infinitely many original clusters which do not
engage in collisions before any given time $t>0$.
This family of lambda-coalescents can be characterized by the
moment condition
\begin{equation}\label{fin-nu}
\int_0^1 x\,\nu(\diff x)<\infty.
\end{equation}
We call the collection of
singleton clusters of $\Pi_\infty(t)$  {\it the dust
component}.
The dust component has
a positive total frequency, meaning that  the number of
singletons within $\Pi_n(t)$ grows approximately linearly in $n$
as $n\to\infty$.

The coalescents with dust component do not exhaust all coalescents which stay infinite.
One distinguished example is the Bolthausen-Sznitman coalescent with $\nu({\rm d}x)=x^{-2}{\rm d}x$ which stays infinite
although (\ref{fin-nu}) fails. Such examples on the border between `coming down from infinity' and `possessing dust component'
are more of an exception if one considers e.g. measures $\nu$ satisfying a condition of regular variation near zero.

Under (\ref{fin-nu}) every transition of $\Pi_\infty$ will involve infinitely many singletons. This suggests that most of the
collision events of $\Pi_n$ will
involve some of the original $n$ clusters, for large $n$. Another way to express this idea is to say
that in a tree representing the complete merging history
of $\Pi_n$, most of the internal nodes are linked
directly to one of $n$ leaves.
We will show that this intuition is indeed correct, to the extent that the behaviour of $\tau_n$ and $X_n$ can be derived from
that of analogous quantities associated with the evolution of the dust component.
In turn,
the total frequency of the dust component of
$\Pi_\infty$  undergoes a relatively simple  process, which may be   represented as $\exp(-S_t)$
where $S=(S_t)_{t\geq 0}$ is a subordinator.
Similarly for  $\Pi_n$,
the engagement of original $n$ clusters  in their first collisions
follows a Markovian process which has been studied in the context of regenerative composition structures
derived from subordinators \cite{GnePit}.
A coupling of
$\Pi_\infty$ with $S$ will enable us
to apply known results about the level-passage for subordinators, and about the asymptotics of
 regenerative composition structures.

The connection between $\Pi_\infty$ and $S$ was first explored in \cite{GneIksMoe} in the special case  when $\nu$ is a finite measure,
hence subordinator $S$ is a compound Poisson process.
While in the present paper we are mainly interested in infinite $\nu$,
the case of finite $\nu$ is not excluded. Moreover, we will be able to extend the results of \cite{GneIksMoe}
by removing a condition on $\nu$
imposed in that paper.

In a recent paper by Haas and Miermont \cite{haas}
results on counting collisions in the coalescent and counting blocks in the regenerative composition
were  derived  separately in the context of absorption times of decreasing Markov chains.
Our approach adds some insight to the connection between these two models, and  it entails some delicate features like differentiating
between collisions which involve some original clusters of $\Pi_n$ and the collisions which do not.

\par The possible modes of
behaviour of $\tau_n$ and $X_n$ for large $n$  are best
illustrated by the family of coalescents driven by a beta measure
\begin{equation}\label{be}
\nu(\diff x )= c x^{a-3} (1-x)^{b-1}\diff x,
~~~~~a,b,c>0.
\end{equation}
These coalescents come down from infinity for $a<1$ and stay
infinite for $a\geq 1$. With the account of  results of the present
paper, we have the following list.

\begin{itemize}

\item[(i)] Case $0<a<1$.
The limit law of $(X_n-(1-a)n)/n^{1/(2-a)}$ is $(2-a)$-stable (see \cite{GneYak}, and \cite{IksMoe2} for the case $b=1$).
 The distribution of $\tau_\infty$ is unknown.

\item[(ii)] Case $a=1$. The instance
 $b=1$ is the Bolthausen-Sznitman coalescent, for which
 the limit distribution of $\tau_n-\log \log n$ is standard Gumbel
\cite{FreMoe, gold}, while $X_n$, suitably scaled and centered,
converges weakly to a $1$-stable distribution \cite{DIMR2,IksMoe}.
The case $b\neq 1$ remains open.

\item[(iii)] Case $1<a< 2$.
In the sequel we show that $(\tau_n-c\log n)/(\log n)^{1/2}$ (with suitable $c>0$) converges weakly to a normal distribution,
and that $X_n/n^{2-a}$ converges to the exponential functional of a subordinator.
The result about $X_n$ was proved previously in  \cite{haas}, and in
\cite{IksMoe2} in  the case $b=1$.

\item[(iv)] Case $a\geq 2$. In the present paper we prove that
normal limits hold for both $\tau_n$ and $X_n$ with explicitly
determined scaling and centering. In the case $a>2$ these
asymptotics were previously shown in \cite{GneIksMoe}. In the case
$a=2$ the result for $X_n$ was derived in \cite{IksMarMoe}.
\end{itemize}

\section{The coalescent and singleton clusters}

In the role of the state space of the coalescent $\Pi_n$ with initially $n$ clusters
we take the set  of  partitions of  $[n]:=\{1,\dots,n\}$, in which every singleton cluster is classified as either
{\it primary} or {\it secondary}. Under the {\it dust component of}  $\Pi_n(t)$ we shall understand the collection of primary
clusters.
Every
nonsingleton cluster of $\Pi_n(t)$ is regarded as secondary.
For the notational convenience the clusters are written by increase of their minimal elements,  the elements within the clusters are written in
increasing order, and the secondary clusters are written in brackets.
For instance,  $1~ (2)~ (3~ 5~ 6)~  4~ 7$, a partition of the set $[7]$, has three primary clusters and two secondary:
$1,4,7$ and  $(2), (3~5~6)$, respectively.

Introduce $\lambda_{m,1}$ as in \eqref{lrates} with $k=1$. We have
$\lambda_{m,1}<\infty$ by assumption \eqref{fin-nu}.

We define the lambda-coalescent $\Pi_n$ as a c{\'a}dl{\'a}g Markov
process with values in such partitions of $[n]$ and the initial
state $1~2~\cdots~n$ with $n$ primary clusters. Each admissible
transition is either merging some clusters in one cluster, or
turning a primary
singleton cluster into secondary.
From partition with $m$ clusters, the transition rate
 for merging each particular $k$-tuple of $m$ clusters in one is $\lambda_{m,k}$ ($2\leq k\leq m$),
and the transition rate for turning  each particular primary
singleton cluster into secondary singleton cluster is $\lambda_{m,1}$. For instance,
the sequence of distinct states visited by $\Pi_7$  could be
\begin{multline}\nonumber
1~   2~   3~   4~   5~   6~ 7  \to   1~ 2~ (3~ 5~ 6)~  4~ 7
\to   1~ (2)~ (3~ 5~ 6)~  4~ 7
\to
\\
1 ~(2~ 4)~ (3~ 5~ 6)~ 7 \to
1 ~ (2~3~4~5~6~7) \to (1~2~3~4~5~6~7).
\end{multline}

Let $N_n(t)$ be the number of clusters in $\Pi_n(t)$. Then  $N_n=(N_n(t))_{t\geq 0}$ is a nonincreasing Markov process,
with the transition
rate
\begin{equation}\label{phirate}
\varphi_{m,k}:={m\choose k}\lambda_{m,k}
\end{equation}
for jumping from $m$ to $m-k+1$, for $2\leq k\leq m$. Turning a
primary
singleton cluster into a secondary
singleton cluster does not cause a jump of $N_n$. The absorption
time of $\Pi_n$ can be recast as $\tau_n=\inf\{t: N_n(t)=1\}$, and
the number of collisions $X_n$ is equal to the number of jumps the
process $N_n$ needs to proceed from $n$ to $1$ (which is 4 in the
above example where the second transition  does not alter the
number of clusters).

Removing element $n$ transforms partition of $[n]$ into partition
of $[n-1]$. For example, partitions $1~(2~4)~(3)$, $1~(2)~3~4$ and
$1~(2)~3~(4)$   all become $1~(2)~(3)$. Restricting in this way
$\Pi_n$ to $[n-1]$, pointwise in $t\geq 0$, yields a stochastic
copy of $\Pi_{n-1}$. This follows as in \cite{Pit} since the rates
satisfy the recursion
$\lambda_{m,k}=\lambda_{m+1,k}+\lambda_{m+1,k+1}$ for all $1\leq
k\leq m$. Therefore we may define $\Pi_n$ on the same probability
space consistently for all $n$. Explicit realization will appear in the sequel.

The projective limit of the processes $\Pi_n,~ n\in {\mathbb N},$ is a Markov process $\Pi_\infty$
starting at $t=0$ with the infinite configuration of primary clusters $1~2~\cdots$, and assuming
values in the space of partitions of the infinite set ${\mathbb N}$.
Each partition $\Pi_\infty(t)$ has only primary singletons, namely those original clusters
which do not engage in collisions up to time $t$.
For a generic singleton, e.g. labelled $1$, the time before its first collision
has exponential distribution with parameter $\lambda_{1,1}$, and when such collision occurs  infinitely many other
clusters are engaged.

The differentiation of singletons of $\Pi_n(t)$ into  primary and secondary   becomes transparent
by considering $\Pi_n$ as restriction of $\Pi_\infty$ on $[n]$.
The secondary  singletons of $\Pi_n(t)$
 are the unique representatives in $[n]$
of some infinite clusters of $\Pi_\infty(t)$.
The primary singletons of $\Pi_n(t)$ are also singletons in the partition $\Pi_\infty(t)$.


There is a construction of $\Pi_\infty$ based on a planar Poisson point process in the strip $[0,1]\times [0,\infty)$ with
intensity measure $\nu(\diff x)\times \diff t$, see \cite{BerestyckiLec, BertoinLec, Pit}.
 With each atom $(t,x)$
one associates a transition of $\Pi_\infty$ performed by tossing
a coin with probability $x$ for heads.
To pass from  $\Pi_\infty(t-)$ to $\Pi_\infty(t)$, the coin is tossed for each cluster of $\Pi_\infty(t-)$, then those clusters marked
heads are merged in one, while the clusters
marked tails remain unaltered. Although 
there are infinitely many transitions within any time interval 
if $\nu$ is an infinite measure, condition (\ref{fin-nu}) 
ensures that $\Pi_\infty$ does not terminate. In the case of finite $\nu$ transitions
of $\Pi_\infty$ occur at the epochs of Poisson process with rate $\nu([0,1])$.

Let $N_n^*(t)$ be the number of primary clusters in $\Pi_n(t)$. By
homogeneity properties of $\Pi_n$, the process
$N_n^*=(N_n^*(t))_{t\geq 0}$ is  a nonincreasing Markov process,
jumping at rate $\varphi_{m,k}$ from $m$ to $m-k$ for $1\leq k\leq
m$. Let 
$$\tau^*_n:=\inf\{t: N_n^*(t)=0\}$$
 be the random time
when the last of $n$ primary clusters disappears. For $1\leq r\leq n$,
let $K_{n,r}$\label{knr} be the number of
decrements of size $r$ of $(N_n^\ast)$ on the way from $n$ to $0$, let
$K_n:=\sum_{r=1}^n K_{n,r}$ be the total number of decrements of
$(N_n^\ast)$, and let $X_n^*$ be the number of non-unit decrements of
$(N_n^*)$. Obviously,
\begin{equation}\label{xxxx}
X_n^\ast=K_n-K_{n,1}.
\end{equation}
We call 
the clusters of partition $\Pi_n^*(\tau_n^*)$ that remain at time
$\tau_n^*$ {\it residual}, and we denote $R_n$ the number of residual 
clusters.

\par  Processes  $N_n$ and $N_n^*$ look very similar, thus at a first glance it might
seem surprising that $N_n^*$ is much easier to handle. The
simplification comes from the identification of the sequence of
decrements of $N_n^*$
with the $n$th level of a regenerative composition structure \cite{GnePit}, and further connection to
the range of a subordinator.
The main new contribution of the present paper is
that  $N_n^*$ yields a good approximation for $N_n$ for large
$n$, thus $X_n^*$ and $\tau_n^*$ are close to their counterparts
$X_n$ and $\tau_n$. In one direction, the connection is quite
obvious:
$$X_n^*\leq X_n, ~~~N_n^*(t)\leq N_n(t),~~~\tau_n^*\leq \tau_n.$$
For instance, the first inequality holds since every collision
taking at least two primary clusters contributes to $X_n$, and
since with positive probability some $R_n\geq 2$ clusters  remain at time $\tau_n^*$ when the last primary clusters
disappears.

\section{Coupling  with a subordinator}

Condition \eqref{fin-nu} implies that there exists a  pure-jump
subordinator $S=(S_t)_{t\geq 0}$ with the Laplace transform
\begin{equation}\label{LK}
{\mathbb E}(e^{-z S_t})=e^{-t\Phi(z)},~~~z\geq 0,
\end{equation}
where the Laplace exponent is given by
$$\Phi(z):=\int_0^1 (1-(1-x)^z)\nu(\diff x).$$
The coalescent process will be represented in terms of passage of $S$ through multiple exponentially distributed levels.
We describe first  the evolution  of the dust component.

Let $\eps_1,\dots,\eps_n$ be independent of $S$ i.i.d. standard exponential  random variables,
and let $\eps_{n:n}<\dots<\eps_{n:1}$ be their order statistics.
It is not difficult to see that $\Phi(n)$ coincides with the probability rate at which
the subordinator
passes through the level $\eps_{n:n}$
from any state $S_t=s <\eps_{n:n}$.
The following lemma extends this observation.

\begin{lemma} For $t\geq 0$, conditionally given $S_t=s$ with $s\in (\eps_{n:m+1},\eps_{n:m})$ the subordinator is passing through
$\eps_{n:m}$ at rate $\Phi(m)$, and is hitting at this passage each of the
intervals $(\eps_{n:m-k+1},\eps_{n:m-k})$ at  rate $\varphi_{m,k}$,
for $1\leq k\leq m\leq n$.
\end{lemma}
\begin{proof} The proof exploits the L{\'e}vy-Khintchine formula \eqref{LK} and the memoryless property of the exponential distribution.
See computations around \cite[Theorem 5.2]{GnePit} for details.
\end{proof}

Now suppose that each of the primary clusters $1~2\dots n$ is
given an exponential mark $\eps_1,\dots,\eps_n$, and that for
every $t\geq 0$ the marks $\eps_j>S_t$ are associated with primary
clusters $j$ existing at time $t$. If $t$ is a jump-time of $S$
and the interval $(S_{t-},S_t]$ covers exactly one mark $\eps_j$,
we interpret the  event of passage through $\eps_j$ as turning the
primary cluster $j$ into secondary. If $(S_{t-},S_t]$ covers at
least two of the $\eps_j$'s, we interpret this event as a
collision which takes the corresponding primary clusters. Setting
$N_n^*(t):=\#\{j\in [n]: \eps_j>S_t\}$ we obtain a process with
desired rates $\varphi_{m,k}$ for transition from $m$ to $m-k$, as it
follows from the lemma. In particular,
$\Phi(n)=\sum_{k=1}^n\varphi_{n,k}$ coincides with the total
transition rate of the coalescent $\Pi_n$ from the initial state
$1~2\cdots~n$.

A {\it regenerative} ordered partition of the set $[n]$ is defined by sending $i,j$ to the same block iff $T_{\eps_i}=T_{\eps_j}$, see  \cite{GnePit}.
The number of blocks of the partition is equal to the number of jumps of $N_n^*$ prior to the absorption at state $0$.

These evolutions of primary clusters
are consistent in $n$.
Assigning the exponential marks $\eps_1,\eps_2,\dots$ to infinitely many primary clusters
$1~2~\dots$ defines the initial state of the dust component. The frequency of the dust component of $\Pi_\infty$ as time passes is the decaying
process $(\exp(-S_t))_{ t\geq 0}$.

\par One straightforward application of the representation by $S$ concerns $\tau_n^*$, the maximal lifetime
of primary clusters in $\Pi_n$. Let
$$T_s:=\inf\{t\geq 0: ~S_t>s\}$$
be the first passage time through level $s\geq 0$. We can identify
$\tau_n^*$ with $T_{\eps_{n:1}}$, hence connect the limit
behaviour of $\tau_n^*$ to that of $T_s$ for high levels $s$.
Indeed, from  the extreme-value theory  it is known that
$\eps_{n:1}-\log n$ converges in distribution, as $n\to\infty$, to
a random variable with the Gumbel distribution. It is also known
that the scaled and centered random variables $(T_s-g(s))/f(s)$
can converge in distribution only if the normalizing constant
$f(s)$ goes to $\infty$ with $s$. Thus, $T_{\eps_{n:1}}$ and
$T_{\log n}$ have the same limit law, if any. Moreover, it can be
shown that $(T_s-g(s))/f(s)$ converges weakly to a given proper
and nondegenerate probability law if and only if the same  holds
for $(T^\prime_s-g(s))/f(s)$, where $T^\prime_s$ is the number of
points within $[0, s]$ of a random walk which starts at $0$ and
has the generic step distributed  like $S_1$. See \cite{Bing} (or
Proposition 27 in \cite{Neg}) for a complete list of limit
distributions of $T^\prime_s$ and the conditions of convergence.
Summarizing the above, we have
\begin{proposition}\label{propo1} For constants $a_n>0$ and  $b_n\in{\mathbb R}$,
if one of the random variables  $(\tau_n^*-b_n)/a_n$ and
$(T_{\log n}-b_n)/a_n$
converges weakly,  as
$n\to\infty$, to a nondegenerate proper distribution,
then the other random variable converges weakly to this distribution too.
\end{proposition}


To realize the full dynamics of $\Pi_n$ in terms of the level-passage, a mark is  assigned to each cluster
according to the following rule.
At time $0$ the marks $\eps_1,\dots, \eps_n$ represent the primary  clusters $1~2\cdots n$.
At time $t>0$ there is some collection of marks on $[S_t,\infty)$ representing the clusters existing at this time.
If at time $t>0$ the subordinator passes through exactly $k$ marks corresponding to some
clusters $I_1,\dots,I_k\subset[n]$, then a new cluster $I_1\cup \dots\cup I_k$ is born and assigned a mark $S_t+\eps$,
where $\eps$ is a copy of the unit exponential random variable, independent of $S$ and all other marks
assigned before $t$.
For instance,
if at the first passage time $t=T_{\eps_{n:n}}$ the subordinator jumps through exactly $k$ levels $\eps_{j_1},\dots,\eps_{j_k}$
out of $\eps_1,\dots,\eps_n$, then the secondary cluster $J=\{j_1,\dots,j_k\}$
is born (which is a singleton if $k=1$)
and assigned a mark exponentially distributed on $[S_t,\infty)$.

In particular,  when $S$ passes at some time $t$ through only one mark, there is no change in $\Pi_n(t)$, and the
mark of the corresponding singleton cluster is just re-assigned.




\section{The absorption time}

We wish to exploit the lifetime $\tau_n^*$ of primary clusters as approximation to the absorption time $\tau_n$.
At time $\tau_n^*$ the coalescent process is left with $R_n$ residual
clusters, whence the distributional identity
\begin{equation}\label{recursion_tau}
\tau_0:=0, \ \ \tau_n\stackrel{d}{=} \tau_n^* + \tilde{\tau}_{R_n}, \ \
n\in\mn,
\end{equation}
where $\tilde{\tau}_m$ is assumed independent of $(\tau_n^*,R_n)$
and distributed like $\tau_m$, for each $m\in\mn_0$.
To address the quality of approximation we need to estimate $R_n$.

We begin with some preparatory work.
By the first transition the Markov chain $N_n^*$ goes from $n$
to a  state with distribution
 $p_{n,k}:=\varphi_{n,n-k}/\Phi(n)$, $0\leq k\leq n-1$.
Let  $g_{n,k}$ be the probability that $N_n^*$ ever visits state
$k$, so in terms of the realization via  subordinator,
$g_{n,k}={\mathbb P}( T_{\eps_{n:k+1}}<T_{\eps_{n:k}})$ is the probability
that the interval $[\eps_{n:k+1},\eps_{n:k}]$ intersects the range
of $S$. An explicit formula for $g_{n,k}$ in terms of $\Phi$ is
available (see \cite{GnePit}, Equation (50)), but it is complicated and
inconvenient for computations.

\begin{lemma}\label{g_sum_est}
Suppose $(r_k)_{k\in\mn}$ is a nonnegative sequence such that the
sequence $\left(\frac{\Phi(k)r_k}{k}\right)_{k\in\mn}$ is
nonincreasing. Then the sequence $(a_n)_{n\in\mn_0}$ defined  by
$$
a_0=0,\;\;a_n:=\sum_{k=1}^{n}g_{n,k}r_k,\;~~~n\geq 1
$$
satisfies the relation
$$a_n=O\Big(\sum_{k=1}^n \frac{r_k\Phi(k)}{k}\Big),\;\;n\to\infty.$$
\begin{proof}

The assertion follows from  Lemma \ref{boundedness} in the
Appendix. Indeed, conditioning on the size of the first jump of
$N_n^*$ we see that the sequence $(a_n)$ satisfies the recurrence
$$
a_0=0, \ \ a_n=r_n+\sum_{k=0}^{n-1} p_{n,k}  a_{k}, \ \ n\in\mn.
$$
To apply Lemma \ref{boundedness}
we take $\psi_n=\Phi(n)$. Condition (C2) holds by
the assumptions and condition (C1) follows from
\begin{equation}\label{cond2_checking}
\Phi(n)\sum_{k=0}^n(1-k/n)p_{n,k}=\frac{1}{n}\sum_{k=0}^{n-1}(n-k)\varphi_{n,n-k}=\frac{1}{n}\sum_{k=1}^n
k\varphi_{n,k}=\int_0^1 x\nu({\rm d}x)>0.
\end{equation}
\end{proof}
\end{lemma}


Note that, since the function $s\mapsto \Phi(s)/s$ is nonincreasing,
the sequence $\left(\frac{\Phi(k)r_k}{k}\right)$ is nonincreasing
whenever $(r_k)$ is itself nonincreasing.

Denote $\vec{\nu}(x):=\nu([x,1])$, $x\in (0,1)$.

\begin{lemma}\label{ma}
If either of two equivalent conditions
\begin{equation}\label{ex1} \int_0^1
x^{-1}\diff x\int_0^x \vec{\nu}(y){\rm d}y<\infty,
\end{equation}
\begin{equation}\label{im1}
\sum_{k=1}^\infty{ \Phi(k)\over k^2}<\infty
\end{equation}
holds then
$${\mathbb E} R_n=O(1),~~~n\to\infty,$$
in which case  the sequence of distributions of the $R_n$'s is
tight.
\end{lemma}
\begin{proof} The equivalence of \eqref{ex1} and \eqref{im1} is established by
repeated integration by parts.

\par  In the genealogical history of each residual cluster there is the  last
secondary cluster appearing as a result of collision or switch
involving some primary clusters.
If   secondary cluster $b$ is born at some time $t\leq\tau_n^*$ of such an event,
and if at this time
some $j\geq 0$ other primary clusters co-exist, then $b$ corresponds
to a residual cluster provided that $b$ and its followers do not collide with
these $j$ primary clusters or their followers before time $\tau_n^*$.
That is to say, $b$ and the $j$ primary clusters belong to
distinct branches if the coalescent tree is cut at time
$\tau_n^*$. Let $q_j$ be the probability that such cluster $b$
 corresponds to a residual cluster; restricting the coalescent to $j+1$ clusters
it is seen that
$q_j$ indeed depends only on $j$.
The consistency property of
the coalescent with respect to the restrictions entails that
$q_j$ is decreasing in $j$. Averaging over the times when
some primary clusters engaged we find
the expected number of residual clusters
\begin{equation}\label{r_n_mean}
{\mathbb E}R_n=\sum_{j=0}^{n-1}g_{n,j} q_j.
\end{equation}

Furthermore,
given $S_t=s$, we have exactly $j$ exponential marks of the primary clusters larger than $s$.
The cluster $b$ is assigned a new exponential  mark $u=s+\eps$
which  lies within each of the spacings in $(s,\infty)$ generated
by $\eps_{n:j},\ldots,\eps_{n:1}$  with the same probability
$1/(j+1)$. If this spacing  is $(\eps_{n:k+1},\eps_{n:k})$ then
$b$ {\it may} correspond to a residual cluster only if (i)
$T_{\eps_{n:k+1}}<T_{u}<T_{\eps_{n:k}}$ and  (ii) $b$  does not
collide further with $k$ primary clusters and their followers
before time $\tau_n^*$. 
If (i) occurs, condition (ii) is not sufficient for the correspondence 
since possible collisions with some of $j$ primary clusters or their followers 
are ignored.
This leads to the inequality
\begin{eqnarray*}
q_j &\leq& {1\over j+1}\sum_{k=0}^{j} g_{j+1,k+1} p_{k+1,k} q_k,
~~~1\leq j\leq n-1,
\end{eqnarray*}
and  $q_0=1$. Substituting
$\varphi_{k,1}=k(\Phi(k)-\Phi(k-1))$ we obtain

\begin{eqnarray*}
q_j
& \leq&  {1\over j+1}\sum_{k=1}^{j+1} g_{j+1,k} \frac{k(\Phi(k)-\Phi(k-1))}{\Phi(k)}q_{k-1}\\
    &\leq& {c\over j+1}\sum_{k=1}^{j+1}(\Phi(k)-\Phi(k-1))q_{k-1},\\
\end{eqnarray*}
where Lemma \ref{g_sum_est} was applied with
$$r_k=\frac{k(\Phi(k)-\Phi(k-1))q_{k-1}}{\Phi(k)}.$$
The required monotonicity condition holds since both $q_k$ and
$\Phi(k)-\Phi(k-1)$ are decreasing in $k$, the latter by concavity of $\Phi$.
Here and throughout $c$ will denote a positive constant whose value is not important and may change from line to line.

Setting $a_j=(j+1)q_j$ and $b_j=c({\Phi(j+1)-\Phi(j))/(j+1)}$,
we obtain from the above
$$a_j\leq \sum_{k=0}^{j}b_k a_k,\;\;\;j\in\mn_0.$$
We want to show that the sequence $(a_j)$ is bounded. To that end,
let  $M_j:=\max_{i=0,\ldots,j}a_i$, then also $$ M_j\leq
\sum_{k=0}^{j}b_kM_k.
$$
Since $\Phi(j)/j$ decreases we have $\Phi(j+1)-\Phi(j)\leq
\Phi(j+1)/(j+1)$, which taken together with
 \eqref{im1}
implies that the series $\sum_{k=0}^{\infty}b_k$ converges, so we
can choose $$n_0:=\inf\{k\geq 0 : \sum_{i=k}^{\infty}b_i<1/2\}.$$
If $\lim_{n\to\infty}M_n=\infty$ then
$$
1\leq \underset{n\to\infty}{\lim\inf}\, \frac{\sum_{k=0}^n
b_kM_k}{M_n}=\underset{n\to\infty}{\lim\inf}\,\frac{\sum_{k=n_0}^n
b_kM_k}{M_n}\leq \sum_{k=n_0}^\infty b_k\leq 1/2,
$$
which is an obvious contradiction. Therefore $(a_n)$ is bounded.
From this $$q_j\leq M_j/(j+1)\leq c/j.$$Substituting this bound into \eqref{r_n_mean} and
applying Lemma \ref{g_sum_est} leads to the conclusion that
${\mathbb E} R_n$ remains bounded, as $n\to\infty$, by the virtue of
\eqref{im1}.
\end{proof}

Recall that the convergence of $T_s$ in distribution always requires a scaling constant going to $\infty$ as $s\to\infty$.
Under conditions of  Lemma \ref{ma} the sequence of laws of $\tau_{R_n}$ is tight.
Now from  Proposition \ref{propo1} and the decomposition
\eqref{recursion_tau}
the following main result of this section emerges.

\begin{thm}\label{abstime} Suppose {\rm (\ref{ex1})} holds.
For some constants $a_n>0$  and  $b_n\in{\mathbb R}$, if
 one of the variables $(T_{\log n}-b_n)/a_n$ and
$(\tau_n-b_n)/a_n$ converges weakly, as $n\to\infty$, to a nondegenerate proper
distribution then the other variable converges weakly to the same
distribution.
\end{thm}

The value of this result lies in the fact that the limit laws for $T_s$ and the conditions of convergence
are immediately translated into the convergence of $\tau_n$.
Normalizing and
centering constants are known explicitly, see
Proposition 27 in
\cite{Neg} or
\cite{Bing}.
It follows that only stable laws and the Mittag-Leffler laws
can appear as the limit distributions of $\tau_n$.

If measure $\nu$ is finite the
condition \eqref{ex1} obviously holds.
In this case $S$ is a compound Poisson process.
Theorem \ref{abstime} has  been
proved \cite{GneIksMoe} under the assumptions that
 $\nu$ is not supported by a geometric
sequence $(1-x^k)_{k>0}$  (meaning that the law of $S_1$ is nonlattice) and that
\begin{equation}\label{log_moment}
\theta:=
\int_0^1|\log x|\,\nu({\rm d}x)<\infty.
\end{equation}
Theorem \ref{abstime} shows
that the result of \cite{GneIksMoe} is still true without requiring (\ref{log_moment}).

Assumption \eqref{ex1} is not very restrictive since
$\Phi(k)=o(k)$, $k\to\infty$, always holds.
Concretely, suppose the right tail of $\nu$ has the property of regular variation at $0$, that is
\begin{equation}\label{RV}
\vec{\nu}(x)\sim x^{-\gamma}\ell(1/x), ~~~x\downarrow 0,
\end{equation}
for some function $\ell$ of slow variation at $\infty$, and
$\gamma\in [0,1]$.
Then condition \eqref{ex1} is satisfied for
$\gamma\in [0,1)$. In the edge case $\gamma=1$ the behaviour of
$\ell$ is important, for instance \eqref{ex1} holds for
$\ell(y)=(\log y)^{-\delta}$ if $\delta>2$ and does not hold if
$\delta\in (1,2]$.

We use  condition \eqref{ex1} to bound $R_n$, although we perceive that \eqref{ex1} can be omitted and
the equivalence in Theorem \ref{abstime} holds in full generality for the coalescents with dust component.
Note that \eqref{ex1} is the local property of $\vec{\nu}$ near $0$.
More substantially, the limit law is affected
by the decay at $\infty$ of the right tail of the distribution of $S_1$,
for which the behaviour of $\vec{\nu}$ near $1$ is responsible. We illustrate this by two examples.

\vskip0.3cm
\noindent
\textbf{Example: normal limits.}
{\rm Assume in addition to \eqref{ex1} that
$${\tt s}^2:={\rm
Var}(S_1)=\int_0^1 |\log(1-x)|^2 \nu(\diff x)<\infty.$$
Then as $n\to\infty$
\begin{equation}\label{nor1}
{\tau_n-{\tt m}^{-1}\log n\over ({\tt m}^{-3}{\tt s}^2 \log
n)^{1/2}} \ \stackrel{d}{\to} \,{\cal N}(0,1),
\end{equation}
where ${\tt m}:=\me S_1=\int_0^1 |\log(1-x)|\nu(\diff x)$.

This setting applies to beta coalescents mentioned in
Introduction. We choose the constant in (\ref{be}) to be
$c=1/{{\rm B}(a,b)}$, where ${\rm B}$ is the beta function. 
The case $a>2$ was settled in \cite{GneIksMoe}.
We focus on the previously open case
$1<a\leq 2$.

 For $a=2$ we compute the constants as
$${\tt m}=b(b+1)\zeta(2,b), \ \ {\tt s}^2=2b(b+1)\zeta(3,b),$$
where $\zeta$ is the Hurwitz zeta function.

For $a\in (1,2)$ we have
\begin{eqnarray*}
{\tt m}={a+b-1\over
(a-1)(2-a)}\bigg(1-(a+b-2)\{\Psi(a+b-1)-\Psi(b)\}\bigg),\\
\end{eqnarray*}
\begin{eqnarray*}
{\tt s}^2=
{a+b-1\over (a-1)(2-a)}\times~~~~~~~~~~~~~~~~~~~~~~~~~~~~~~~~~~~~~~~~~~~~~~~~~~~~~~~~~~~~~~ \\
\bigg(2\{\Psi(a+b-1)-\Psi(b)\}
-(a+b-2)\{(\Psi(a+b-1)-\Psi(b))^2+\Psi^\prime(b)-\Psi^\prime(a+b-1)\}\bigg),
\end{eqnarray*}
where $\Psi$ is the logarithmic derivative of the gamma function.
Finally, condition \eqref{ex1} holds since
\eqref{RV} is satisfied with $\gamma=2-a\in [0,1)$ and constant function $\ell$.
Therefore, convergence \eqref{nor1} holds with the computed ${\tt m}$ and ${\tt
s}$.}

\vskip0.3cm
\noindent
\textbf{Example: stable limits.}
 Assume \eqref{ex1} and
\begin{equation}\label{domain1}
\vec\nu(1-e^{-y})\sim y^{-\beta}L(y), ~~\\ y\to \infty,
\end{equation}
for some function $L$ slowly varying at $\infty$ and
$\beta\in(1,2)$. Then
\begin{equation}\label{nor2}
{\tau_n-{\tt m}^{-1}\log n\over {\tt m}^{-(\beta+1)/\beta}c_{\lfloor\log
n\rfloor}} \stackrel{d}{\to} \,{\cal S}(\beta),~~ \ \ n\to\infty,
\end{equation}
where $c_n$ is any sequence satisfying
$\lim_{n\to\infty}nL(c_n)/c_n^\beta=1$, and ${\cal S}(\beta)$ is
the $\beta$-stable distribution with characteristic function
      $$
z\mapsto \exp\{-|z|^\beta\Gamma(1-\beta)(\cos(\pi\beta/2)
      +i\sin(\pi\beta/2)\, {\rm sgn}(z))\},\quad z\in\mr.
      $$

To illustrate, consider
$$\nu(\diff x)={x^{a-2}\diff x\over (1-x)|\log(1-x)|^d}\,,$$ where
$d\in (2,3)$ and $a\in (d, d+1)$. Then \eqref{fin-nu} is
satisfied, and condition \eqref{RV} holds with $\gamma=d+1-a\in
(0,1)$ which implies \eqref{ex1}. Condition
\eqref{domain1} is fulfilled with $\beta=d-1\in (1,2)$. Therefore, the
absorption time $\tau_n$ of such coalescent
has limit law \eqref{nor2}.

\section{The number of collisions}

\subsection{Preliminaries}

As an approximation to the number of collisions $X_n$ we shall
consider $X_n^*$, the number of jumps of $N_n^*$ of size at least
two. We will not be able to derive a complete result comparable with  Theorems
\ref{main} or \ref{abstime} because the universal criterion for
convergence of  $X_n^\ast$ is not available. The cases when we
know the behaviour of $X_n^*$ (from \cite{GnePitYor1, GnePitYor2},
\cite{GBarbour} and \cite{GneIksMar}) are all covered by the
assumption that $\nu$ satisfies (\ref{RV}).  We shall also proceed
in this direction but exclude the case $\gamma=1$
when $K_{n,1}$ is the term of dominating growth in the sum $K_n=\sum_{r=1}^n K_{n,r}$. 
By Karamata's
Tauberian theorem \cite{BGT} condition \eqref{RV} with $\gamma<1$
is equivalent to the analogous asymptotics of the Laplace exponent
$$\Phi(z)\sim \Gamma(1-\gamma) z^\gamma \ell(z),~~~z\to\infty.$$
The case of finite $\nu$ appears when $\gamma=0$ and $\Phi$ is an increasing bounded function.


The sequence $(X_n)$ is nondecreasing and satisfies a
distributional recurrence
\begin{eqnarray}\label{recur1}
X_1=0, \ \ X_n\stackrel{d}{=} \tilde{X}_{n-J_n+1}+1(J_n\geq 2), \ \
n\geq 2,
\end{eqnarray}
where in the right-hand side $J_n$ is independent of the
$\tilde{X}_i$'s, $\tilde{X}_i\stackrel{d}{=} X_i$, and $J_n$ is
distributed like the first decrement of $N_n^*$, that is ${\mathbb
P}(J_n=k)=p_{n,n-k}$ for $1\leq k\leq n$. Similarly, the number
$X_n^*$ of collisions which involve at least two primary clusters
satisfies
\begin{eqnarray}\label{recur2}
X_1^\ast=0, \ \ X_n^*\stackrel{d}{=} \tilde{X}_{n-J_n}^*+1(J_n\geq
2), \ \ n\geq 2,
\end{eqnarray}
with the convention $X_0^\ast=0$. We may decompose
$X_n$ as
\begin{equation}\label{repr}
X_n=X_n^\ast+D_n\overset{\eqref{xxxx}}{=} K_n-K_{n,1}+D_n \ \
\text{a.s.}
\end{equation}
where $D_n$ is the number of collisions which take at most one primary cluster.
Thus a collision contributes to $D_n$ if either exactly one primary cluster merges with at least one secondary cluster,
or  at least two secondary and no primary clusters are merged.
\begin{lemma}\label{ma120}
We have
\begin{equation}\label{estimD}
{\mathbb E} D_n\leq c \sum_{k=1}^n \left({\Phi(k)/ k}\right)^2, \
\ n\in\mn.
\end{equation}
In particular, if either of two equivalent conditions
\begin{equation}\label{ex11} \int_0^1
x^{-2}\bigg(\int_0^x \vec{\nu}(y){\rm d}y\bigg)^2{\rm d}x<\infty,
\end{equation}
\begin{equation}\label{im11}
\sum_{k=1}^\infty \left({\Phi(k)/k}\right)^2<\infty
\end{equation}
holds then the sequence of distributions of the $D_n$'s is tight.
\end{lemma}
\begin{proof}
The equivalence of \eqref{ex11} and \eqref{im11} follows from
\cite[Proposition 1.4]{Bert}.

Choose some primary cluster $b$, to be definite let it be the
cluster labelled 1, and suppose $\tilde{X}_{n-1}$ is realised as
the number of collisions among $n-1$ primary clusters
$[n]\setminus\{b\}$ and their followers. Then
${X}_n=\tilde{X}_{n-1}+Z_n$, where $Z_n$ is the indicator of the
event that the first collision of $b$ involves exactly one other
cluster $a$. At the time of the merge of $b$ with $a$ the Markov
chain $N_n^*$ decrements by two or one, depending on whether $a$
is primary or secondary. Let $Y_n$ be the indicator of the event
that the first involvement of $b$ is either turning $b$ into
secondary cluster, or a collision taking at most one other primary
cluster and arbitrary number of secondary clusters. Clearly,
$Y_n\geq Z_n$, therefore from \eqref{recur1}
\begin{equation}\label{ineq3}
X_n\stackrel{d}{\leq} \tilde{X}_{n-J_n}+Y_{n-J_n+1}+1(J_n\geq 2),
\end{equation}
where $\stackrel{d}{\leq}$ stands for `stochastically smaller'. Passing to
expectations in \eqref{ineq3}, \eqref{recur2} and \eqref{repr} we
see that, for $d_n:={\mathbb E}D_n$, $y_n:={\mathbb E}Y_n$,
$$d_1=0, \ \ d_n\leq \sum_{k=1}^{n-1} p_{n,k}(d_k+y_{k+1}), \ \ n=2,3,\ldots,$$
and iterating this inequality yields
$$d_1=0, \ \ d_n\leq\sum_{j=1}^{n-1}g_{n,j}y_{j+1}, \ \ n=2,3,\ldots$$
By exchangeability, we have $y_n=(\me K_{n,1}+2\me K_{n,2})/n$.
Since $$\me K_{n,1} =
\sum_{k=1}^n g_{n,k}p_{k,k-1}=
\sum_{k=1}^n
g_{n,k}{k(\Phi(k)-\Phi(k-1))\over \Phi(k)},$$ using Lemma
\ref{g_sum_est} with $r_k={k(\Phi(k)-\Phi(k-1))/\Phi(k)}$ yields
\begin{equation}\label{121}
\me K_{n,1}\leq c\,\Phi(n), \ \ n\in\mn.
\end{equation}
Using this, an inequality shown in Appendix and the monotonicity
of $\Phi$,
$$\me
K_{n,2}\overset{\eqref{122}}{\leq} c_1\, \me K_{\lceil
n/2\rceil,1}\overset{\eqref{121}}{\leq} c_2\, \Phi(\lceil
n/2\rceil)\leq c_2\,\Phi(n).$$ Thus
$$d_n\leq c\,\sum_{k=1}^n g_{n,k} \Phi(k)/k,$$
and using Lemma \ref{g_sum_est} with $r_k=\,c\Phi(k)/k$ results in
\eqref{estimD}.
\end{proof}

\subsection{The compound Poisson case}
Assume that $\nu$ is a finite measure on $(0,1)$, not supported by
a geometric sequence of the form $(1- x^k)_{k\geq 0}$, for some
$x\in (0,1)$. Since a linear time change of the coalescent does
not affect the distribution of $X_n$ we will not lose generality
by assuming that $\nu$ is a probability measure on $(0,1)$. Let
$(W_k)_{k\in\mn}$ be independent copies of a random variable $W$
such that the law of $1-W$ is $\nu$. The subordinator $S$ is then
a unit rate compound Poisson process with the generic jump $|\log
W|$ having some nonlattice law.

The variable $K_{n,r}$ introduced on
p.~\pageref{knr} can be identified with the number of parts of size $r$ in
the regenerative composition of $[n]$ associated with $S$.
Alternatively, $K_{n,r}$ has interpretation in terms of the following occupancy model (see e.g. \cite{GneIksMar}).
Consider a random discrete distribution
$$P_k:=W_1\ldots W_{k-1}(1-W_k), \ \ k\in\mn,$$
with $P_k$ thought of as a frequency of box $k$.  Suppose $n$ balls are thrown in infinitely many
boxes, independently given $(P_k)$, with probability $P_k$ of falling in box $k$  for each ball. Then $K_{n,r}$ can be identified with the number of boxes
occupied by exactly $k$ out of $n$ balls.

Introduce
$${\tt m}:=\int_0^1 |\log(1-x)|\nu({\rm d}x)$$
and for $1\leq r\leq n$ let $\varkappa_{n,r}:=\me K_{n,r}$.

\begin{assertion}\label{singl}
\vskip0.01cm
\noindent
\begin{itemize}
\item[\rm (a)] If  ${\tt m}<\infty$ then for every $r=1,2,\dots$ 
  the vector $(K_{n,1}, K_{n,2},\ldots, K_{n,r})$ converges weakly, as $n\to\infty$, 
to a proper multivariate distribution, and 
$\varkappa_{n,r}\to({\tt m}r)^{-1}.$
\item
[\rm (b)] If ${\tt m}=\infty$ then $\varkappa_{n,r}\to 0$, so $K_{n,r}\to 0$ in probability.
\end{itemize}
\end{assertion}
\begin{proof} Part (a) was proved in  \cite{GIR}, Theorem 3.3.

For (b) consider a random walk $(Q_j)_{j\geq 0}$  with $Q_0=0$ and
the generic step $|\log W|$. Then $P_j=(1-W_j)\exp(-Q_{j-1})$.
Using $1-x\leq e^{-x}$ with $x\in[0,1]$, and substituting $e^z$
for $n$, we reduce estimating $\varkappa_{n,1}=n\sum_{j\geq
1}P_j(1-P_j)^{n-1}$ to estimating

\begin{eqnarray*}
\me \left(\sum_{j\geq 1}  e^z P_je^{-e^z P_j}\right) &=&\me
\left(\sum_{j\geq
1}(1-W_j)\exp\{z-Q_{j-1}-e^{z-Q_{j-1}}(1-W_j)\}\right)\\&=&\int_0^\infty
f(z-y){\rm d}U(y),
\end{eqnarray*}
where $f(y):=\me \{(1-W)\exp(y-e^y(1-W))\}$ and $U(y):=\sum_{j\geq
0}\mmp\{Q_j\leq y\}$ is the renewal function of the random walk.
The function $f$ is nonnegative and integrable, since
$\int_{-\infty}^\infty f(y){\rm d}y=1$. Furthermore, the function
$y\to e^{-y}f(y)$ is nonincreasing. It is known
 that  these properties together
ensure that $f$ is directly Riemann integrable (see, for instance,
the proof of Corollary 2.17 in \cite{DurLiggett}). When ${\tt
m}=\me |\log W_j|=\infty$, application of the key renewal theorem
yields $\int_0^\infty f(z-y)U({\rm d}y)\to 0$, as $z\to\infty$,
whence $\varkappa_{n,1}\to 0$.

For $r>1$ the argument is similar, or one can use the estimate
$\varkappa_{n,r}\leq c_r \varkappa_{n,1}$ shown in Appendix, Lemma
\ref{kar}.
\end{proof}

The next theorem improves upon a result from
 \cite{GneIksMoe} by removing condition \eqref{log_moment}.
\begin{thm}\label{main}
For constants $a_n>0$ such that $\lin a_n=\infty$, and
$b_n\in\mr$, whenever
any of the variables
$${K_n-b_n\over a_n},~~~{X^\ast_n-b_n\over a_n}~~~ {\rm or~~} {X_n-b_n\over a_n}$$
converges weakly, as $n\to\infty$, to a nondegenerate proper distribution then all three
variables converge weakly to this distribution.
\end{thm}
\begin{proof}
Recall representation \eqref{repr}. Since $\nu$ is a probability
measure we have $\Phi(k)<1$, hence condition \eqref{im11} is
satisfied, and the
 sequence of laws of the $D_n$'s is tight by Lemma
\ref{ma120}. By Proposition \ref{singl}, the sequence of laws of
the $K_{n,1}$'s is tight as well. By the assumption
$a_n\to\infty$ the result follows.
\end{proof}

From \cite{GneIksMar} it is known that, depending on the behaviour
of $\vec{\nu}(x)$ near $x=1$
there are five different modes of the
weak convergence of, suitably normalized and centered, $K_n$. We
do not exhibit all these cases here, rather provide an example
borrowed from \cite{GneIksMar} to demonstrate a substantial role
of the parameter $\theta=\int_0^1 |\log x|\nu(\diff x)$.
\vskip0.3cm
\noindent
{\bf Example}. Suppose $\nu$ has the right tail of the form
\begin{equation*}\label{the}
\vec{\nu}(x)={|\log x|^\rho \over 1+|\log x|^\rho}, ~~~~ x\in
(0,1]
\end{equation*}
with $\rho>0$.
In the case $\rho\in (0,1/2)$ we have $\theta=\infty$, and
$${X_n-{\tt m}^{-1}\log n+({\tt m}(1-\rho))^{-1}\log^{1-\rho} n\over c \log^{1/2} n} \ \stackrel{d}{\to} \ {\cal N}(0,1), \ \
n\to\infty,$$ where ${\tt m}=\int_0^1 |\log(1-x)|\,\nu({\rm d}x)$.

In the other case,
when $\rho>1/2$ (then $\theta<\infty$ for $\rho>1$),
the centering simplifies, so that
$${X_n-{\tt m}^{-1}\log n\over c \log^{1/2} n} \ \stackrel{d}{\to}
{\cal N}(0,1), \ \ n\to\infty.$$

\vskip0.3cm
\noindent
\paragraph{Evolution of secondary particles.} In the compound Poisson case the number $V_t$ of secondary clusters of $\Pi_\infty(t)$
is finite, for each $t\geq 0$. The process $V=(V_t)_{t\geq 0}$
starts with $V_0=0$ and is a Markov chain with the transition rate
$\varphi_{m,k}={m\choose k} \lambda_{m,k}$ for jumping from $m$ to
$m-k+1$, $0\leq k\leq m$, $k\neq 1$. The rate for $k=0$ is given
by the same formula (\ref{lrates}), and $\varphi_{m,0}<\infty$
because $\nu$ is finite. The $k=0$ transition, resulting in
increase of the number of secondary clusters by one, occurs when
some (in fact, infinitely many) primary clusters
 merge without engagement of secondary
clusters. The stationarity of $V$ is a consequence of the existence of  the dust component with infinitely many clusters.

It can be shown that
the Markov chain $V$ is positively recurrent and has a unique stationary distribution $(\pi_m)$ found from
 the balance equation
\begin{equation}\label{balance}
\pi_m=\sum_{k=0}^\infty \pi_{m+k-1}\varphi_{m+k-1,k}\,
\end{equation}
supplemented by the conditions $\pi_0=0$ and $\sum_{m=1}^\infty \pi_m=1$.


Suppose for example that $\nu({\rm d}x)={\rm d}x$ is the Lebesgue measure on $[0,1]$. In this case $\varphi_{m,k}=(m+1)^{-1}$.
Equation (\ref{balance}) becomes
$\pi_m=\sum_{j=m-1}^\infty {\pi_{j}/( j+1)}.$
Differencing yields
$\pi_m-\pi_{m+1}={\pi_{m-1}/m},$
which is readily solved as
$$\pi_m={e^{-1}\over (m-1)!}\,,~~~~m=1,2,\dots$$
so in this case the stationary distribution  is shifted Poisson.

In contrast, the number of secondary clusters in the finite coalescent $\Pi_n$ is not a Markov process, because
the transition rates  depend on the number of remaining primary particles.


\subsection{The case of slow variation}
Suppose (\ref{RV}) holds with $\gamma=0$ and slowly varying $\ell(z)\to\infty, ~~z\to\infty$.
The Laplace exponent satisfies then $\Phi(z)\sim\ell(z)$.
Suppose also that the subordinator has finite moments
$${\tt m}=\me S_1=\int_0^1 |\log(1-x)|\nu(\diff x),~~~{\tt s}^2={\rm Var}\, S_1=\int_0^1 |\log(1-x)|^2\nu(\diff x).$$
Choose the centering/scaling constants as
$$b_n= {1\over{\tt m}}\int_0^n {\Phi(z)\over z}\diff z\,,~~~a_n=\sqrt{{{\tt s}^2\over{\tt m}^3}\int_0^n {\Phi^2(z)\over z}\diff z}.$$
In \cite{GBarbour} it was shown that for  $n\to\infty$
$$\me K_n\sim b_n, ~~~\sqrt{{\rm Var }K_n}\sim a_n,$$
and that the normal limit $(K_n-b_n)/a_n\stackrel{d}{\to}{\cal
N}(0,1)$ holds for various classes of functions $\ell$. In
particular, this includes functions of slow variation at infinity
with asymptotics as diverse as
$$\ell(z)=\log(\log( \dots(\log (z))\dots)), ~~~\ell(z)=\log^\beta z ~,~~~\ell(z)=\exp(\log^\beta z),$$
where $\beta>0$.

The series (\ref{im11}) converges for arbitrary $\ell$, hence
by Lemma \ref{ma120}
$\me D_n=O(1)$. On the other hand, from (\ref{121}) and by the properties of slowly varying functions \cite{BGT}
$$\me K_{n,1}=O(\Phi(n))=o(a_n).$$
It is immediate now that $(K_n-b_n)/a_n\stackrel{d}{\to}{\cal
N}(0,1)$ implies both $(X_n^*-b_n)/a_n\stackrel{d}{\to}{\cal
N}(0,1)$ and $(X_n-b_n)/a_n\stackrel{d}\to{\cal N}(0,1)$.
\vskip0.3cm \noindent {\bf Example: gamma subordinators.} Consider
the classical  gamma subordinator with Laplace exponent
$\Phi(z)=\alpha \log(1+z/\beta)$, where $\alpha,\beta>0$. The
corresponding $\nu$ driving the coalescent has density
$$\nu(\diff x)={\alpha(1-x)^{\beta-1}\over|\log(1-x)|}\,\diff x.$$
The central limit theorem for $K_n$ was proved by different methods in \cite{GnePitYor1} and \cite{GBarbour}.
From this we conclude that
 the number of collisions also satisfies $(X_n-b_n)/a_n\stackrel{d}\to{\cal N}(0,1),$ where the constants can be chosen as
$$a_n=\sqrt{{\beta\log^3 n\over 3}}~  ,~~~~b_n={\beta\log^2 n \over 2}.$$
\vskip0.3cm
\noindent
{\bf Example:
beta$(2,b)$-coalescents.}
For this family
$\nu(\diff x)=x^{-1}(1-x)^{b-1}\diff x$. The convergence of $X_n$ to the standard normal distribution
holds with scaling/centering constants
$$a_n=\sqrt{\frac{{\tt s}^2}{3{\tt m}^3}\log^3 n},~~~b_n={\log^2 n\over 2{\tt m}},$$
where ${\tt m}=\zeta(2,b),~{\tt s}^2=2\zeta(3,b)$.

\subsection{Regular variation with index $0<\gamma<1$.}
A key distribution in this case is the law of the random variable
$$I=\int_0^\infty \exp(-\gamma S_t)\diff t,$$
known as the exponential functional of the subordinator $\gamma S$.
The distribution of $I$ is uniquely determined by the moments
$$\me I^k={k!\over\prod_{i=1}^k\Phi(\gamma i)}.$$
From \cite{GnePitYor2} (Theorem 4.1 and Corollary 5.2) $X_n^*/a_n\stackrel{d}{\to} I$,
where $a_n=\Gamma(2-\gamma)n^\gamma\ell(n)$, and no centering is required.
In fact, $K_n/a_n$ and $K_{n,r}/a_n$ ($r\geq 1$) converge almost surely and in the mean.

To justify the convergence of $X_n$ using (\ref{repr}) we need to etimate $\me D_n$.
For $0<\gamma<1/2$ we have $\me D_n=O(1)$ since $\Phi(z)\sim c\,\ell(z)z^\gamma$, hence  the series \eqref{im11} converges.
For $1/2<\gamma<1$ we have
$$\sum_{k=1}^n\left( \Phi(k)/k\right)^2\sim c\, n^{2\gamma-1}\ell^2(n),$$
and for $\gamma=1/2$ 
the latter sum, as  a function of $n$,
has the property of
slow variation at infinity (see \cite{BGT}, Proposition 1.5.8).
Thus in any case $D_n/a_n\to 0$ in probability. It follows that
$X_n/a_n\stackrel{d}{\to}I$.

\vskip0.3cm
\noindent
{\bf Example:
beta$(a,b)$-coalescents with $1<a<2$.}
In this case
 $$ {X_n \over n^{2-a}}\ \stackrel{d}{\to} \ {\Gamma(a+b)\over (2-a)\Gamma(b)} \int_0^\infty
   \exp\{{{-(2-a)S_t}\}}\,\diff t, \ \ n\to\infty.
$$
This result was obtained in \cite{haas} by  another
method, and
with a change of variables the equivalence
with Theorem 7.1 from \cite{IksMoe2} in the case $b=1$ can be established.

The subfamily of beta-coalescents with parameters $b=2-a$ was intensively studied. In the literature sometimes $\alpha:=2-a$ is taken as parameter, so that $\nu$ in this notation 
becomes
$$\nu(\diff x)=x^{-\alpha-1}(1-x)^{\alpha-1}.$$
In this case $N_n^*$ decrements like a random walk conditioned to hit $0$ and, moreover, there is an explicit formula (see \cite{GnePit} p. 471)
$$g_{n,k}={(\alpha)_k(\alpha)_{n-k}\over(\alpha)_n}{n\choose k},$$
where $(\alpha)_k$ denotes the rising factorial.
The variable $K_n$ is then the number of blocks in Pitman's $(\alpha,\alpha)$-partition
(or in the regenerative composition induced by excursions of a Bessel bridge \cite{GnePit}).
We refer to \cite{CSP} and \cite{BerestyckiLec} for further multiple connections  of these beta-coalescents
to various random processes.

\section{Appendix}

\paragraph{A linear recursion.}
For each $n\in {\mathbb N}$, let $(p_{n,k})_{0\leq k\leq n}$ be a
probability distribution with $p_{n,n}<1$. Define a sequence
$(a_n)_{n\in\mn}$ as a (unique) solution to the recursion
\begin{equation}\label{recursion}
a_n=r_n+\sum_{k=0}^n p_{n,k}a_k,~~~~n\in \mn,
\end{equation}
with given  $r_n\geq 0$ and the initial value  $a_0=a\geq 0$.
\begin{lemma}
\label{boundedness} Suppose there exists a sequence
$(\psi_n)_{n\in \mn}$ such that
\begin{itemize}
\item[\rm(C1)]
$\liminf_{n\to\infty}\psi_n\sum_{k=0}^{n}(1-k/n)p_{n,k}>0$,
\item[\rm(C2)] the sequence $(\psi_k r_k/k)_{k\in\mn}$ is
non-increasing.
\end{itemize}
Then  $(a_n)$ defined by \eqref{recursion} satisfies
\begin{equation}\label{bounded}
a_n=O\Big(\sum_{k=1}^{n}\frac{r_k\psi_k}{k}\Big),\;\;n\to\infty.
\end{equation}
In particular, $(a_n)$ is bounded if
the series $\sum_{k=1}^{\infty}\frac{r_k\psi_k}{k}$ converges.
\end{lemma}

\begin{proof} Write for simplicity $p_{k}$ for $p_{n,k}$ and let $\pi_k=\sum_{j=0}^k p_j$.
Using (C2) we have
$$\sum_{k=1}^n{r_k\psi_k\over k}\pi_{k-1}\geq {r_n\psi_n\over n}\sum_{k=1}^n\pi_{k-1}=r_n\psi_n\sum_{j=0}^{n-1}(1-j/n)p_{j}.$$
By (C1) there exist $n_0\in\mn$ and $c>0$ such that
\begin{equation}\label{fromC1}
c\sum_{k=1}^n {r_k\psi_k\over k}\pi_{k-1}\geq r_n\,,~~~~n\geq n_0.
\end{equation}
From this, $x_n:=c\sum_{k=1}^n r_k\psi_k/k$ satisfies
\begin{equation}\label{recursion1}
x_n\geq r_n+\sum_{k=1}^n x_kp_k\,,~~~n\geq n_0
\end{equation}
To check the latter, write
\begin{eqnarray*}
r_n+\sum_{k=1}^n x_kp_k&=&r_n+c\sum_{j=1}^n\sum_{k=j}^n {r_j\psi_j\over j}p_k\\
&=& r_n+c\sum_{j=1}^n {r_j\psi_j\over j}(1-\pi_{j-1})\\
&=&
r_n+c\sum_{j=1}^n {r_j\psi_j\over j}-c \sum_{j=1}^n {r_j\psi_j\over j}\pi_{j-1}\\
& = & x_n+r_n-c \sum_{j=1}^n {r_j\psi_j\over
j}\pi_{j-1}\overset{\eqref{fromC1}}{\leq }x_n.
\end{eqnarray*}
Set $x_0:=0$. Subtracting \eqref{recursion} from
\eqref{recursion1} we see that $y_n:=x_n+c_0-a_n$ satisfies
$y_n\geq \sum_{k=0}^n p_k y_k$ for $n\geq n_0$ and arbitrary
$c_0$.
We can achieve that the recursion for $y_n$ holds for all
$n\in\mn$ by choosing $c_0\geq \max_{n\leq n_0}a_n$. But then it
is easily shown by induction that $y_n\geq 0$ for all $n\in\mn$,
which implies the desired estimate of $a_n$.
\end{proof}

\vskip0.3cm
\noindent
\paragraph{Estimates for the occupancy counts.}
Let $(p_k)_{k\in\mn}$ be a probability mass function.
Consider the multinomial occupancy scheme in which $n$ balls are thrown independently in boxes, with probability $p_j$ for box $1,2,\dots$
The expected number of boxes occupied by exactly $r$ out of $n$ balls is
$$\varkappa_{n,r}={n\choose r} \sum_{j\geq 1}
p_j^r(1-p_j)^{n-r}, \ \ 1\leq r\leq n.$$

\begin{lemma}\label{kar}
For fixed $r<s$ there exists a constant $c$ such that
\begin{equation}\label{122}
\varkappa_{n,r}\geq c\, \varkappa_{2n,s}, \ \ n\in\mn.
\end{equation}
\end{lemma}
\begin{proof}
Using $(1-x)^{-1}\geq e^x$ for
 $x\in (0,1)$,
\begin{eqnarray*}
{n \choose r}x^r(1-x)^{n-r}\over{2n\choose
s}x^s(1-x)^{2n-s}&\geq& c_1 {s!\over 2^s r!}(nx)^{r-s}(1-x)^{s-r-n}\\
&{\geq}& c_2
(nx)^{r-s} e^{nx/2}\\
&\geq& c_2
\min_{y>0} y^{r-s}e^{y/2}\\
&=& c_2
\left({e\over 2(s-r)}\right)^{s-r}.
\end{eqnarray*}
\end{proof}

The result extends immediately to the case of random $(P_k)$.
This generalization was used in the proof of Proposition
\ref{singl} with the $P_j$'s being the sizes of intervals obtained
by splitting $[0,1]$ at points of the range of the process
$\exp(-S)$.

\vskip0.4cm
\noindent {\bf Acknowledgement} A. Iksanov gratefully
acknowledges the support by a grant from the Utrecht University.

\end{document}